\documentclass{miseEnPage}

\usepackage{etoolbox}
\usepackage{amsmath,amsfonts,enumerate}

\newtheorem{theorem}{Theorem}[section]
\newtheorem{corollary}{Corollary}[theorem]
\newtheorem{lemma}[theorem]{Lemma}

\newtheorem{assu}[theorem]{Assumption}
\theoremstyle{definition}

\theoremstyle{proposition}
\newtheorem{proposition}[theorem]{Proposition}
\newtheorem{remark}[theorem]{Remark}

\numberwithin{equation}{section}

\usepackage{hyperref}

\usepackage[margin=1in]{geometry}
\usepackage{macros_perso}
\begin{document}
   \title[Tear-off versus global existence for an age-structured model of adhesion]
   {Tear-off versus global existence for a structured model of adhesion mediated by transient elastic linkages.
 }


\author{Vuk Mili\v si\' c}
 \address{Laboratoire Analyse, G{\'e}om{\'e}trie \& Applications (LAGA),
Universit{\'e} Paris 13,
FRANCE}
\email{milisic@math.univ-paris13.fr}
\thanks{The first author was granted by Campus France ({\tt www.campusfrance.org/}) in the framework of the project 27238 TD}

\author{Dietmar Oelz} 
\address{Courant Institute, New York University}
\email{dietmar@cims.nyu.edu}
\thanks{ This study has been supported by the Wolfgang Pauli Institute (Vienna) and by the Vienna Science and Technology Fund (WWTF) through its project MA09-004; furthermore through the Austrian Agency for International Cooperation in Education 
and Research (OeAD) through its project FR 08/2012 and by the Schr\"odinger Fellowship J3463-N25 of the Austrian Science Fund (FWF).}

\date{\today}


\begin{abstract}
We consider a microscopic model for friction mediated by transient elastic linkages introduced
in \cite{MiOel.1,MiOel.2}. In the present study we prove existence and uniqueness of a solution to
the coupled system under weaker hypotheses. The theory we present covers the case where the off-rate
of linkages is unbounded but increasing at most linearly with respect to the mechanical load.

The time of existence is typically bounded and corresponds to tear-off where the moving binding site does not have any
bonds with the substrate. However, under additional assumptions on the external force we prove global in time 
existence of a solution that consequently stays attached to the substrate.
\end{abstract}

\maketitle

\setcounter{tocdepth}{1}
\tableofcontents

\section{Introduction}

Adhesion forces at the cellular and intra-cellular scales play an important role 
in several  phenomenons  as cell motility (see \cite{OeSchVi} and references therein), 
or cancer growth \cite{PreVi}. In \cite{OeSchVi,OeSch} the authors derive a complete
model for a moving network of actin filaments  polymerizing near the boundary of the cell 
and depolymerizing close to the nucleus, providing biologically plausible steady-state configurations 
of the cell shape. 
The main advantage of this method is that the parameters used are experimentally 
easy to obtain if not already available in the literature \cite{pmid18278037,pmid8432732,pmid8282102,pmid12547805,pmid1093925}.
The adhesion and the stretching between filaments are written as friction
terms obtained through a formal  limits of a delayed system of equations.
Indeed, let $\e$ be a dimensionless parameter denoting  the
ratio of the typical lifetime of bonds vs. the overall timescale of the model.
The asymptotic limit is obtained assuming that the   rate of linkage turnover becomes great 
as well as  the stiffness of the bonds (typically as $O(1/\e)$).
The rigorous justification of the limit $\e\to 0$ is the ultimate goal of our investigations \cite{MiOel.1,MiOel.2}.
Nevertheless the highly non-linear nature of the delayed model leads  to
consider    already the case of  a fixed value of $\e$.  
In this article we show that, even then, the data of the problem
determins the well-posedness of the model :
the balance between the on-rate of the linkages and the force exerted on the adhesion point
is essential.  Mathematically this is seen since, according to this balance, either we can show 
blow up in finite time  or global existence. Physically this means that
pulling too strong the adhesion site causes a tear-off, and that our model is able 
to reproduce this feature. 
Experimentally this  is observed and used in order to measure the probability distribution of  unbinding forces  \cite{Baumgartner11042000,canetta:inserm-00144609,pmid16183875,hanley2003single}.

More precisely, this study is concerned with a system of equations which describe
the evolution of the time-dependent position of a single binding site as it moves on a 1D-substrate.
External forces $f$ act on the moving point-object while it is attached to the substrate through continuously remodeling
elastic linkages which represent the effect of transiently attaching protein bonds. Their age distribution is denoted by $\rhoe=\rhoe(t,a)$ where
$a \geq 0$ denotes the age of linkages and $t \geq 0$ denotes time. 
Here we treat $\e$ as a fixed constant, which we keep in our notation to maintain consistency with previous studies, in which we were analyzing the convergence with respect to $\varepsilon$ \cite{MiOel.1,MiOel.2}. 

In \cite{OeSchVi, OeSch} the following structured model for the turnover of protein bonds with
age distribution $\rhoe=\rhoe(t,a)$ for age $a \geq 0$ has been established,
\begin{equation}
\label{eq.rho.eps} 
\left\{ 
\begin{aligned} 
&\varepsilon \partial_t \rhoe + \partial_a \rhoe+ \zteps \, \rhoe = 0 
\,, &t>0 \, , \;a > 0 \; , \\ 
& \rhoe(t,a=0)=\beps(t)\left(1-\muze  \right) 
\, , & t > 0 \; , \\
& \rhoe(t=0,a)=\rhoi(a)
\, ,&a\geq 0  \; , 
\end{aligned}  
\right. 
\end{equation}
where $\muze(t):=\int_0^\infty \rhoe(t,\tilde a) \, d \tia$ and the on-rate of bonds is a given
coefficient $\beps$ times a factor, that takes into account saturation of the moving binding site with linkages.
This system is coupled to the  elongation $\veps=\veps(t,a)$ of the linkage through 
\begin{equation}
\label{def_full_zteps}
\zteps:=\xi(|\veps(t,a)|) \; ,
\end{equation} 
according to which the off-rate $\xi(u) > 0$ is a real, positive function of the elongation of the linkage $\veps$.
In \cite{MiOel.2} we introduced the following age-structured model for the evolution of the elongation $\veps$,
\begin{equation}
\label{eq.veps}
\left\{ 
\begin{aligned}
& \e \dt \veps + \da \veps = \frac{1}{\muze} \left(  \e \dt f + \int_0^\infty \left(\zeta (\veps) \veps \rhoe\right)(t,\tia) \; d\tia \right) \,, &t>0 \, , \;a > 0 \; , \\
& \veps(t,a=0)= 0\,, &t>0  \; , \\
& \veps(t=0,a)=\vepsi(a)\,, &a \geq 0  \; ,
\end{aligned}
\right.
\end{equation}
where $f=f(t) \in \mathbb{R}$ is the external force acting on the binding site.
In \cite{MiOel.2} we had shown that the system \eqref{eq.veps} is equivalent to an integral equation for
the position of the binding site $\zeps=\zeps(t)$ itself \cite{MiOel.1}, namely
\begin{equation}
\label{eq.z.eps} 
\left\{ 
\begin{aligned} 
& \frac{1}{\varepsilon} \int_0^\infty \left(z_\varepsilon(t)-z_\varepsilon(t-\varepsilon a)\right) \rho_\varepsilon(t,a) \; da = f(t) \; , \quad & t\geq 0 \; , \\
& z_\varepsilon(t)=\zp(t) \; , \quad  &t < 0 \; ,
\end{aligned}  
\right. 
\end{equation}
where the known past positions are given by the Lipschitz function $\zp(t) \in \mathbb R$ for $t<0$. 
The correspondance between \eqref{eq.z.eps} and \eqref{eq.veps} is made through 
\begin{displaymath}
 \veps=\frac{\zeps(t)-\zeps(t-\varepsilon a)}{\varepsilon} \; .
\end{displaymath}
Note that the equation \eqref{eq.z.eps} has been the original result of the mathematical modeling since it represents the balance of external and elatic forces acting on the binding site. On the other hand, in \cite{MiOel.2} it has turned out to be beneficial to work on 
the system \eqref{eq.rho.eps}, \eqref{eq.veps} instead, since it allowed to derive powerful a-priori estimates on $\veps$.

The analysis in the older studies \cite{MiOel.1} and \cite{MiOel.2} relied on the existence of an upper bound $\ztmax$ of the function $\zeta$. In \cite{MiOel.2}, for fixed $\varepsilon$,  we prove existence and uniqueness of 
weak solutions of the coupled system \eqref{eq.rho.eps}, \eqref{def_full_zteps}, \eqref{eq.veps}.
In the simpler semi-coupled case when $\zteps(t,a)$ is a given bounded
function we give as well a convergence result as $\e \to 0$.
The analysis of \eqref{eq.rho.eps} in \cite{MiOel.1} relied on the fact that 
$ \ddt{} \, \cH_0[\rho-\rho_\infty]\leq-\ztmin/\varepsilon \cH_0[\rho-\rho_\infty]$ where
\begin{equation}
\label{equ_Hfunctional}
 \cH_0[\rho](t) := \int_{\rr} | \rho(t,a) | da + \left| \int_{\rr} \rho(t,a) da \right|,
\quad
\end{equation}
and where $\rho_\infty$ is the stationary solution of \eqref{eq.rho.eps}.
For the analysis of \eqref{eq.veps} in \cite{MiOel.2} we established the a-priori estimate 
(Lemma 4.1,  \cite{MiOel.2})
\begin{equation}
\label{equ.stab.rho.u}
\int_{\rr} \rhoe (t,a) | \veps(t,a) | da \leq 
 \int_{\rr} \rhoi(a) | \vepsi(a) | da + \int_0^t | \dt f(\tit)  | d\tit \; .  
\end{equation}
Note that both results, decay of the functional \eqref{equ_Hfunctional} and the {\em a-priori} estimate \eqref{equ.stab.rho.u}, do not rely on the existence of an upper bound of $\zeta$ and therefore do hold in the framework of this paper.

\bigskip 

It is the aim of the present study to relax the hypothesis of boundedness of $\zteps$.
This represents a major improvement of the analysis, because the lower bound of the total mass $\muze(t)$ strongly depends on $\ztmax$ and the analytical arguments in \cite{MiOel.2} do rely heavily on this control. 
Furthermore the upper bound $\ztmax$ had major importance in the fixed point argument used in \cite{MiOel.2}
to prove the global existence result since we used it to control the non-linear right hand side in \eqref{eq.veps}.

In addition to deepening the analysis, unboundedness of the off-rate is the natural scenario from the application point of view. A typical situation is Bell's law, i.e. an exponential increase of the off-rate 
as the elastic linker is extended, $\zteps=\zeta_0 \exp(|\veps|)$ (cf. \cite{Suda_2001, Li2003}). However, this strongly non-linear scenario is still out of reach of the rigorous mathematical analysis that we present in this study which relies on
$\zeta$ being a (globally) Lipschitz continous function as it's main technical assumption.
The right hand side of \eqref{eq.veps} for a given function $u$,
$$
g_u(t)  := \frac{1}{\mu_{0,u}} \left\{ \e \dt f  + \int_{\rr} \zeta(u(t,a)) \, \rhow_u(t,a) \, u(t,a) \, da \right\} \; ,
$$ 
where $\rhow_u$ solves \eqref{eq.rho.eps} with $\zeta=\zeta(u)$ and $\mu_{0,u} := \int_{\rr} \rhow_u(t,a) \, da$,
can become infinite if either $\mu_{0,u}$ vanishes or $\int_{\rr} \zeta(u) \, u \,\rhow_u \, da$ blows up.
We define the modified right hand side
$$
\tig_u := 
\frac{1}{\max(\mu_{0,u},\umu)} 
\left\{ \e \dt f + \max\left( - \ovp, \min\left(\ovp,\int_{\rr} \zeta(u) \, \rhow_u \, u \,  da \right) \right) \right\} \, ,
$$
where $\umu$ and $\ovp$ are two strictly positive arbitrary constants.
The strategy to prove our existence result is 
first to establish existence and uniqueness of a solution of this modified problem using a fixed point argument in the space
\begin{equation}
\label{def_XT_sharp}
X_T := \left\{ u \in L^\infty_{\loc}((0,T)\times\rr) \st \sup_{t \in  (0,T)} \left\| u(t,a) \, \omega(a)   \right\|_{L^\infty_a} < \infty \right\}
\end{equation}
defined for any specific time $T>0$, where the weight function is
\begin{equation}
\label{def_omega}
\omega(a):=\frac{1}{1+a} \; .
\end{equation}
To this end we introduce the map $\Phi: v \in X_T \mapsto u \in X_T$ where, given $v$, we solve \eqref{eq.rho.eps} with  $\zeta=\zeta(v)$ and  obtain the age distribution $\rho_v$.
Then we look for the solution of the problem : 
\begin{equation}
\label{eq.wueps}
\left\{
\begin{aligned}
& \e \dt u + \da u = \tig_v (t) \; ,& t>0 \; ,\quad a>0\; , \\
& u(t,0)=0\;, &t>0 \; , \\
& u(0,a)=\vepsi(a)\;, &a\geq 0 \; ,
\end{aligned}
\right.
\end{equation}
to obtain $u \in X_T$. 
The right hand side of \eqref{eq.wueps} becomes a bounded function whose bounds 
depend on  the cut-offs $\umu$ and $\ovp$. This allows to prove contraction of the map $\Phi$ on
a time interval that is sufficiently small. Due to the uniform bounds this process can be iterated to obtain
$(\rhow_w,w)$ a unique, global in time, solution.
Then we establish a uniform bound on $p(t):=\int_{\rr} \zeta(w) w \rho_w da$, the second integral term in $g_w$.
This shows that for $\ovp$
sufficiently large with respect to $1/\umu$, $p(t)$ never reaches $\ovp$ so that the solution 
$(\rho_w,w)$ satisfies also a simple-cut-of problem where $\tig_u$ can be replaced by $\titig_u$
defined as~:
 $$
 \titig_u := 
\frac{1}{\max(\mu_{0,u},\umu)} 
\left\{ \e \dt f +\int_{\rr} \zeta(u) \, \rhow_u \, u \,  da \right\} \, .
$$

In a second step, we prove that if additional assumptions hold, this solution never reaches 
the cut-off value $\umu$. 
Otherwise, we give a lower bound to the time span during which the cut-off is not reached. 
In both cases the solution of the modified problem is also the unique solution to the original 
system \eqref{eq.rho.eps}-\eqref{eq.veps} either globally in time or on the finite interval of time.

More precisely, in Section \ref{sec.intermediate}, we analyze the dependence of the lower 
bound of $\mu_{0,u}$ with respect to the $L^\infty(0,T)$ norm of $\titig_u$.
This naturally leads to local existence results for the original 
problem \eqref{eq.rho.eps}-\eqref{eq.veps} in Section \ref{sec.ex.local} by 
providing a minimal time for which 
the solution $(\rho_w,w)$  does not reach the cut-off value $\umu$.
 
Even stronger results are rigorously obtained in Sections \ref{sec.spec.data} and  \ref{sec.blow.up} 
generalising a straightforward computation in the special case where
$\zeta(u)=1+|u|$ and assuming that $\veps$ remains strictly positive. 
In this case, integrating \eqref{eq.rho.eps} in age,  and using the fact that \eqref{eq.z.eps} transforms in $\int_{\rr} \rhoe(t,a) \veps(t,a) da = f(t)$, we obtain that
$$
\e \dt \muze - \beta (1-\muze) + \muze + f = 0 \; ,
$$
which can be solved directly. This provides immediately the bounds
$$
\min \left( \muze(0), \frac{\bmin - \fmax}{\bmax+1} \right) 
\leq \muze(t) \leq \muze(0) \left(1-\frac{t}{t_0}\right), \; 
$$
where
\begin{displaymath}
 t_0
:= \frac{\e}{\bmin+1} \ln\left( 1 + \frac{\muze(0)(\bmin+1)}{\fmin-\bmax}\right)
\end{displaymath}
and leads to a strictly positive lower bound of $\mu_0$ when $\bmin>\fmax$,
whereas if $\fmin >\bmax$, the time $t_0$ is well defined and the binding site tears off, i.e. $\muze(t)$ becomes zero, at $t=t_0$.
These basic ideas provide global existence results (Section~\ref{sec.spec.data}) vs. tear-off results (Section~\ref{sec.blow.up}) under more general assumptions on $\zeta$.

\section{Technical assumptions, preliminary results and a-priori estimates}
\label{sec_prel_apriori}

\subsection{Hypotheses}
 
\begin{assu}\label{ass.zeta}
\begin{enumerate}[a)]
\item There exists a minimal value $\ztmin$  s.t.  $\zeta(u)\geq \ztmin>0$, $\forall u \in \RR$.
\item The derivative of $\zeta$ is bounded {\em i.e. }$\lvert \zeta'(u)\rvert\leq \cz$, $\forall u \in \RR$.
\item The function $f$ is Lipschitz continuous on $[0,T]$ for any positive fixed $T$.
\end{enumerate}
\end{assu}

\begin{remark}
Note that with this definition we do not allow more than a linear growth for $\zeta$. But in contrast to \cite{MiOel.1,MiOel.2}
one has no hypothesis concerning boundedness on $\zteps$.  
\end{remark}

\noindent As in \cite{MiOel.2} we assume also some hypotheses on the initial and boundary data of \eqref{eq.rho.eps}:
\begin{assu}\label{hypo.data.deux}
  The initial condition $\rhoi \in L^\infty_a(\rr)$ is
\begin{enumerate}[(i)]
\item nonnegative, i.e. 
$
\rhoi(a)\geq0 \; ,\quad \text{ a.e. in }  \rr\; .
$
\item Moreover,  the total initial population satisfies
$$
0< \int_0^\infty \rhoi(a) da < 1 \; ,
$$
\item and higher moments are bounded,
$$
0< \int_0^\infty a^p \rhoi (a) \; da \leq  c_p \; ,\quad \text{for } p=1,2   \; ,
$$
where $c_p$ are positive constants depending only on $p$.
\end{enumerate}
\end{assu}
\begin{assu}\label{ass.beta}
For $\beps$ we assume that
\begin{enumerate}[a)]
\item $\beps=\beps(t)$ is a continuous function,  
\item $0< \bmin \leq \beps(t) \leq \bmax$ for all positive times $t$.
\end{enumerate}
\end{assu}
We detail hereafter those results from \cite{MiOel.1} that are still valid in the weaker frame of Assumptions 
\ref{ass.zeta}, \ref{hypo.data.deux} and \ref{ass.beta}.
\begin{theorem} \label{prop.rho.exist} 
We suppose that $u$ is a given function in $X_T$.
Let  Assumptions~\ref{ass.zeta}, \ref{hypo.data.deux} and  \ref{ass.beta}  hold, 
then for every fixed $\varepsilon$ there exists a unique solution 
$\rhow\in C^0(\rr; L^1(\rr)) \cap L^\infty(\rr^2)$ of the problem \eqref{eq.rho.eps}, with the off-rate
$\zeta(t,a):=\zeta(u(t,a))$. 
It satisfies 
\eqref{eq.rho.eps} in the sense of characteristics, namely
\begin{equation}
\label{rho_model_by_characteristics}
\rhow(t,a)=\begin{cases}
\beps(t-\varepsilon a)
 \left(1- \int_0^\infty \rhow(\tilde a,t-\varepsilon a) \, d\tilde a\right) \times&\\
\hspace{2cm}  \times \exp \left(-\int_0^{a}  \zteps(\tilde a, t- \varepsilon (a-\tilde a))\; d \tilde a  \right) \; ,  
& \text{ when } a < t/\varepsilon  \; ,  \\
\rhoi(a-t/\varepsilon) \exp \left(-\frac{1}{\varepsilon} \int_0^{t} \zteps((\tilde t-t )/\varepsilon + a, \tilde t)  \; d \tilde t   \right)  \; ,  &\text{ if }a \geq t/\varepsilon \; .         
        \end{cases}
\end{equation}
\end{theorem}


\begin{lemma}\label{lemxyz} 
Under the same assumptions as in Theorem \ref{prop.rho.exist},
let $\rhow$ be the unique solution of problem \eqref{eq.rho.eps}, 
then it satisfies a weak formulation of this problem, namely
\begin{multline}
\label{equ_rho_weak}
\int_0^\infty\int_0^T \rhow(t,a)  \left( \e \dt \varphi+\da \varphi- \zteps \varphi\right) \; dt \; da  - \e \int_0^\infty \rhow(t,a) \varphi(t=T,a) \; da \; + \\
 + \int_0^T \rhow(t,a=0) \, \varphi (t,0) \; dt  + \e \int_0^\infty \rhoi(a) \varphi(t=0,a) \; da = 0 \; , 
\end{multline}
for every $T>0$ and every test function $\varphi \in \cD([0,T]\times\rr)$.
\end{lemma}

Following the same argumentation as Lemma 2.2 in \cite{MiOel.1} one has
\begin{lemma}
Under the same assumptions as in Theorem \ref{prop.rho.exist}, it holds that  
$ \muze(t) <1$ for any time. This in turn implies that
$\rhow(t,a) \geq 0$ for almost every $(t,a)$ in $\rr^2$.
\end{lemma}
For $p\in\N$ we define the p-th moment of the solution $\rhoe$ of \eqref{eq.rho.eps}
\begin{equation*}
 \mu_p(t) := \int_0^\infty a^p \rhow(t,a) \, da \; .
\end{equation*}
Then, following the same argumentation as Lemma 2.2 in \cite{MiOel.1}, one has
\begin{lemma}
\label{lem.higher.moments_bounds}
Under the same assumptions as in Theorem \ref{prop.rho.exist},
$$
\mu_{p}(t) \leq \mupmax \quad \text{for} \quad p=1,2\; , 
$$
where the generic constants $\mupmax$ read : 
$$
\mupmax := \sum_{\ell =0}^p \frac{p!}{\ell ! \ztmin^{p-\ell}} \mu_\ell(0) + \frac{p!}{\ztmin^p} \frac{\bmax}{\bmin + \ztmin}.
$$
\end{lemma}
\begin{proof}
 When $p=0$ we simply integrate \eqref{eq.rho.eps} with respect to age :
 $$
 \e \dt \muze + \beps \muze + \int_{\rr} \zteps \rhoe da = \beps
 $$
 as $\zteps$ is bounded from below and using Gronwall's lemma one has :
 $$
 \muze(t) \leq \muze(0) + \frac{ \bmax}{\bmin + \ztmin}
 $$
 For any integer $p$ we then write :
 $$
 \e \dt \mu_p + \ztmin \mu_p - p \mu_{p-1} \leq 0
 $$
which using Gronwall's lemma gives 
$$
\nrm{\mu_{p}}{L^\infty(0,T)} \leq \mu_p(0) + \frac{ p}{\ztmin} \nrm{\mu_{p-1}}{L^\infty(0,T)} .
$$
By recursion, one proves the claim.
\end{proof}


\begin{proposition}\label{prop.h.zero}
Under Assumptions \ref{ass.zeta}, \ref{hypo.data.deux} and \ref{ass.beta}, 
setting $\hatrho := \rhow_2-\rhow_1$ where $\rhow_2$ and $\rhow_1$ solve \eqref{eq.rho.eps} 
with off-rates $\zeta(w_2)$ (resp. $\zeta(w_1)$) where $w_2$ (resp. $w_1$) is a function in $X_T$,
we find that
$$
\cH_0[\hatrho](t) \leq c_0 ( 1 - \exp( \ztmin t / \e)) \nrm{\hatw}{X_t}\;, \quad \forall t \in(0,T) \; ,
$$
 where $\hat{w}:=w_2-w_1$, $c_0:=\frac{2}{\ztmin} \cz \mu_{1,\max}$,   $\mu_{1,\max}$ being the bound on the first moment of $\rhow_1$.
\end{proposition}
\begin{proof}
The proof follows the same lines as for Lemma~3.2  and Lemma~3.3 in \cite{MiOel.1} based on the system satisfied by $\hatrho$,
$$
\left\{
\begin{aligned}
&\e \dt \hatrho + \da \hatrho + \zeta_2 \hatrho = - \hat{\zeta} \rhow_1 & t>0, a>0,\\
&\hatrho(t,0)= - \beta(t) \int_{\rr} \hatrho(t,\tia) \, d \tia, & t>0,\\
&\hatrho(0,a) = 0, &a>0,\\
\end{aligned}
\right.
$$ 
where $\hat{\zeta}:=\zeta(w_2)-\zeta(w_1)$.
\end{proof}

For $k\geq1$ we define
$$
\cH_k[\rho] := \int_{\rr} (1+a)^k \rho(t,a) da
$$
for these functionals one has :
\begin{proposition}\label{prop.h.k}
Under the same hypotheses as in the previous proposition, and  if moreover 
 $$
 \int_{\rr} (1+a)^\ell \rhoi(a) \, da < \infty, \quad \forall \ell \in \{ 0,k+1\}, 
 $$
 then 
$$
\cH_k[\hatrho](t) \leq h_k ( 1 - \exp( -\ztmin t / \e)) \nrm{\hatw}{X_t}, \quad \forall t \in(0,T),
$$
where the constants $h_k$ depend only on $\ztmin,\cz$ and on the constants $(\mu_{\ell,\max})_{\ell \in \{0,k+1\}}$ 
related to the bound on the $\ell$-th  moment
of $\rhow_2$.
\end{proposition}

\begin{proof}
 We apply a recursion argument. The case $k=0$ is proved by Proposition \ref{prop.h.zero}.
 We suppose that the claim is true for $\ell \leq k-1$. We have formally that
 $$
 \e \dt (1+a)^k | \hatrho | + \da  (1+a)^k | \hatrho | - k (1+a)^{k-1} | \hatrho | + \ztmin (1+a)^k  |\hatrho | \leq | \hat{\zeta} | (1+a)^k \rhow_2
 $$
 Integrating in age, one gets that
 $$
 \e \dt \cH_k[\hatrho] - \beta | \hatmu | + \ztmin \cH_k [\hatrho]\leq k \cH_{k-1}[\hatrho] + \cz \nrm{\hatw}{X_t} \int_{\rr} (1+a)^{k+1} \rhow_2(t,a) da
 $$
 which is then estimated giving:
 $$
\e \dt \cH_k[\hatrho]  + \ztmin \cH_k [\hatrho]\leq k \cH_{k-1}[\hatrho] + \cz C_{k+1} \nrm{\hatw}{X_t} + \bmax \cH_0[\hatrho]
$$
which using the Gronwall's Lemma gives
$$
\begin{aligned}
 \cH_{k}[\hatrho] (t) & \leq \frac{1-\exp(-\ztmin t / \e)}{\ztmin} 
\sup_{s \in (0,t)} 
\left( k \cH_{k-1}[\hatrho](s) + \bmax \cH_0[\hatrho](s) + \cz C_{k+1} \nrm{\hatw}{X_s} \right) \\
\end{aligned}
$$
where we used, in the last estimates, the recursion hypothesis and Proposition \ref{prop.h.zero}.
\end{proof}

If we give ourselves $T>0$ and a function $g \in L^\infty(0,T)$ and then we compute 
$w$ as the solution in the sense of characteristics of
 \begin{equation}
 \label{eq.w.eps}
 \left\{
 \begin{aligned}
 & \e \dt w + \da w = g(t) \,, &t>0 \, , \;a > 0 \; , \\
 & w(t,0)=0\,, &t>0  \; , \\
 & w(0,a)=\vepsi(a)\,, &a \geq 0  \; .
\end{aligned}
 \right.
 \end{equation}
 And all along the paper we will assume that the initial condition $\vepsi$ belongs
 to $L^\infty(\rr,\omega)$.
 For this simple transport problem it holds that
 \begin{theorem}\label{thm.w}
If $T>0$  and   $g$ is a function in $L^\infty(0,T)$,  
for any fixed $\e$ and any $T>0$ there exists a unique $w \in X_T$ solving problem \eqref{eq.w.eps}.
Moreover one has the {\em a priori} estimates:
 $$
 \nrm{w}{X_T} \leq \left( \frac{T}{T+\e} \right) \nrm{g}{L^\infty(0,T)} + \nrm{\vepsi}{L^\infty_a(\rr,\omega)}
 $$
 Moreover the maximal time of existence is infinite if $g\in L^\infty(\rr)$.
\end{theorem}

\section{Global existence results for cut-off problems}
\label{sec.global.ex}
We solve the problem find $(\rhow,w)$ satisfying :
\begin{equation}
 \left\{
\begin{aligned}\label{eq.rho.mod}
& \e \dt \rhow + \da \rhow + \zeta(w) \rhoe = 0 \; , &  t>0\, , \; a>0\; ,\\
& \rhow(t,0)=\beps (t)\left(1-\int_{\rr} \rhow(t,a) da\right) \; ,& t>0\; ,\\
& \rhow(0,a) = \rhoi(a)\; ,& a\geq 0\; \\
\end{aligned}
\right.
\end{equation}
and 
\begin{equation}\label{eq.w.mod}
 \left\{
\begin{aligned}
& \e \dt w + \da w = \tig_{w}(t)\;, & t>0\;, \quad a>0 \; ,\\
& w(t,0)=0\; ,& t>0 \; ,\\
&w(0,a)= \vepsi(a) \; & a\geq 0 \; ,\\
\end{aligned}
\right.
\end{equation}
where  we set 
\begin{equation}
\label{eq.g.mod}
\tigw(t) := \frac{1}{\max(\muze(t),\umu)} \left( \e \dt f + \max\left( -\ovp,  \min\left( \int_{\rr}  \zeta(w)\,  \rhoe \, w  \, da\,,\; \ovp \right) \right)\right), 
\end{equation}
where $\muze(t)=\int_{\rr} \rhoe(t,a)\, da$. The two constants  $\umu$ and $\ovp$ are  positive. 

\begin{lemma}\label{lem.lip}
We suppose that $(\umu,\ovp) \in (\RR^*_+)^2$ and that $f$ is Lipschitz. The function 
 $$
 \cG ( A ,B ) := \frac{1}{\max ( \umu,A) }\left\{ \e \dt  f + \max( - \ovp, \min(\ovp,B)) \right\}
 $$
 is a Lipshitz function with respect to $A \in \RR$ for any  fixed $B\in \RR$ and with respect to $B\in\RR$ for any fixed $A  \in \RR$. 
 The Lipschitz constants in both cases are uniform and depend only on $(\umu,\ovp)$.
\end{lemma}

\begin{theorem}\label{thm.sys.mod}
 We suppose that Assumptions \ref{ass.zeta}, \ref{hypo.data.deux} and \ref{ass.beta} hold. Moreover we  assume that 
 $\vepsi \in L^\infty(\rr,\omega)$ and 
 $\nrm{\dt f}{L^\infty(\rr)}$ is finite 
 and that  the constants $\umu$ and $\ovp$ are fixed.
For any fixed time $T$ possibly infinite,  
there exists a unique pair of solutions $(\rhow_w,w)\in C(0,T;L^1(\rr))\times X_T$ 
solving the coupled problems \eqref{eq.rho.mod}, \eqref{eq.w.mod} and \eqref{eq.g.mod}.
\end{theorem}

\begin{proof}
 We apply the Banach fixed point Theorem to  $\Phi$  mapping $w\in X_T \mapsto u \in X_T$ such that
$$
 \left\{
\begin{aligned}
& \e \dt u + \da u = \tigw(t), & t>0, a>0,\\
& w(t,0)=0,& t>0,\\
&w(0,a)= \vepsi(a),& a>0,\\
\end{aligned}
\right.
$$
We prove that $\Phi$ is actually contractive in $X_T$ for a time $T$ small enough.
\begin{enumerate}[a)]
 \item The map $\Phi$ is endomorphic. 
 For any given $w\in X_T$ one has invariably 
\begin{equation}\label{eq.tig}
 | \tigw | \leq \frac{1}{\umu} \left( \e \nrm{\dt f}{L^\infty(0,T)} + \ovp \right) ,
\end{equation}
 which by the same method as in Theorem \ref{thm.w} provides a bound independent on $T$ in $X_T$ on $u$~:
 $$
 \nrm{u}{X_T} \leq \nrm{\tigw}{L^\infty(0,T)} + \nrm{\vepsi}{L^\infty_\omega(\rr)}.
 $$
 \item The map $\Phi$ is a contraction.  
 We set $\hatgw := \tig_{w_2} -  \tig_{w_1}$ and $\hatrho := \rhow_{w_2} - \rhow_{w_1}$ and so on. 
 Thanks to Lemma \ref{lem.lip}
 $$
 |  \hatgw(t) |  \leq \frac{|\hatmu |}{\umu^2} \left\{ \e \nrm{\dt f}{L^\infty(0,T)} + \ovp \right\} + \frac{1}{\umu} \left| \widehat{\left( \int_{\rr} \zeta \rhow w da \right)} \right| =: I_1 + I_2.
 $$
$I_1$ is immediately estimated thanks to Proposition \ref{prop.h.zero}, and one has :
$$
I_1 \leq \frac{1}{\umu^2}   \left\{ \e \nrm{\dt f}{L^\infty(0,T)} + \ovp \right\}  \cH_0[\hatrho](t)
\leq \frac{1}{\umu^2}   \left\{ \e \nrm{\dt f}{L^\infty(0,T)} + \ovp \right\}  c_0 \nrm{\hatw}{X_t},
$$
while we decompose the difference of triple products in $I_2$ as~:
$$
\begin{aligned}
 I_2  &\leq 
 \frac{1}{\umu} 
 \left| \int_{\rr} 
 	\hat{\zeta} \rhow_{w_2} w_2 
	+ \zeta_1 \hatrho w_2 
	+ \zeta_1 \rhow_{w_1} \hatw da 
\right| \\
 & \leq 
 \frac{1}{\umu} 
 \left(
 	\int_{\rr} \cz | \hatw | \rhow_{w_2} | w_2 | da 
	\right. \\
	&\left. + 
	\left( \cz \nrm{w_1}{X_t} + \ztz \right) 
	\left\{ 
		\int_{\rr} (1+a)^2 | \hatrho| da \nrm{w_2}{X_t} 
		+ \int_{\rr} (1+a)^2  \rhow_{w_1} da \nrm{\hatw}{X_t} 
	\right\} 
\right) \\
 & \leq c \left\{ \nrm{\hatw}{X_t} + \cH_2[\hatrho](t) \right\}  \leq \ovc \nrm{\hatw}{X_t}, 
\end{aligned}
$$ 
where the constant $\ovc$ depends on $\cz, \ztz, (\nrm{w_i}{X_t} )_{i\in\{1,2\}}$, $\umu$ and $\int_{\rr} a^k \rhoi(a) da$ for ${k\in\{0,1,2\}}$~.
Using again Theorem \ref{thm.w},  one has 
$$
\nrm{\hatu}{X_t} \leq \frac{t}{t + \e} \nrm{\hatgw}{L^\infty(0,t)} \leq \frac{t}{ \e} \nrm{\hatgw}{L^\infty(0,t)} \leq \frac{t\ovc }{ \e} \nrm{\hatw}{X_t}.
$$
If $T_0< \e/\ovc$ then there exists a unique fixed point $w\in X_{T_0}$ of the mapping $\Phi$. 
\item Global existence for any time.
We suppose that existence is established on the whole time interval $[0,T_{n-1}]$ for $n \geq 1$. We construct a fixed point 
for the next interval $[T_{n-1},T_n := T_{n-1} + \Delta T_n]$ on the map $u=\Phi(v)$
$$
 \left\{
\begin{aligned}
& \e \dt u + \da u = \tig_v(t), & t\in(T_{n-1},T_n) , a>0,\\
& u(t,0)=0,& t\in(T_{n-1},T_n),\\
&u(T_{n-1},a)= w(T_{n-1},a)& a>0.\\
\end{aligned}
\right.
$$
 and
$$
 \left\{
\begin{aligned}
& \e \dt \rho+ \da \rho + \zeta(v) \rho = 0 , & t\in(T_{n-1},T_n), a>0,\\
& \rho(t,0)=\beps (t)\left(1-\int_{\rr} \rho(t,a) da\right) ,&t\in(T_{n-1},T_n),\\
& \rho(T_{n-1},a) = \rhow(T_{n-1},a),& a>0.\\
\end{aligned}
\right.
$$ 
If we denote the extensions to $[0,T_{n}]$ of $(\rho,u)$ as :
$$
\rho_e(t,a) := 
\begin{cases}
\rhow_{\ti{w}}(t,a) & \text{ if } t\in[T_{n-1},T_n) \\
\rhow (t,a) & t\in(0,T_{n-1}]
\end{cases}, \quad
w_e := \begin{cases}
\ti{w}(t,a) & \text{ if } t\in[T_{n-1},T_n) \\
w (t,a) & t\in(0,T_{n-1}]
\end{cases},
$$
where $\ti{w}=\Phi(\ti{w})$ and $(\rhow,w)$ is the unique solution of \eqref{eq.rho.mod}-\eqref{eq.w.mod} on $[0,T_{n-1}]$. 
The continuity of $\rho_e$ allows to apply   Lemma \ref{lem.higher.moments_bounds}.
Similarly  for $w_e$ one has 
$$
\begin{aligned}
 \nrm{w_e}{X_{T_n}} & \leq \nrm{\tigw(t) \chiu{[T_{n-1},T_n)} + \tig_{\weps} \chiu{[0,T_{n-1}]}}{X_{T_n}} + \nrm{\vepsi}{L^\infty_\omega(\RR)} \\
 & \leq \frac{(\e \nrm{\dt f}{L^\infty(0,T_n)} + \ovp)}{\umu} + \nrm{\vepsi}{L^\infty_\omega(\RR)}.
\end{aligned}
$$
where $\chiu{A}$ is the characteristic function of the set $A$, and we used the uniform estimate on $\tig_w$ provided by \eqref{eq.tig}.
These estimates prove that the constant $\ovc$ in the contraction in b) is not changing as time evolves.
Thus we can fix-point again choosing $\Delta T_n$ as in the previous paragraph and prove contraction
in $[T_{n-1},T_n]$. At this step the recursion is complete. The theorem is proven for any positive time.
\end{enumerate}
\end{proof}
\begin{corollary}\label{coro.stab.rho.u.mod} 
Under the same hypotheses as above, 
 for any pair of positive definite reals $(\umu,\ovp)$, the solution-pair $(\rhow_w,w)$ solving \eqref{eq.rho.mod}-\eqref{eq.w.mod} satisfies the {\em a priori} estimates \eqref{equ.stab.rho.u}.
\end{corollary}
\begin{proof}
 We use that 
 $$
\begin{aligned}
  | \tig_w(t) |& \leq \frac{1}{\muze(t)}\left\{\e | \dt f |+ \min\left(\ovp, \left| \min \left( \int_{\rr} \zeta(w) w \rhow_w  da  , \ovp \right) \right|\right)\right\} \\
& \leq \frac{1}{\muze(t)}\left\{\e | \dt f |+ \min\left(\ovp, \int_{\rr} \zeta(w) \rhow_w |w| da\right)\right\}   
 \leq \frac{1}{\muze(t)}\left\{\e | \dt f |+ \int_{\rr} \zeta(w) \rhow_w |w| da \right\}.
\end{aligned}
 $$
 Then  the same arguments as in the proof of Lemma 5.1 in \cite{MiOel.2} provide the {\em a priori} estimates.
 Indeed in the sense of characteristics $| w|$ satisfies :
 $$
 \e \dt | w | + \da | w | \leq | \tig_w | \leq \frac{1}{\muze(t)}\left\{\e | \dt f |+ \int_{\rr} \zeta(w) \rhow_w |w| da \right\}.
 $$
 then multiplying the later inequality by $\rhow_w$ and integrating with respect to age, one gets :
 $$
 \e \dt \int_{\rr} \rhow_w | w | da + \int_{\rr} \zeta(w) | w | \rhow_w da \leq \e | \dt f | + \int_{\rr} \zeta(w) | w | \rhow_w da
 $$ 
 and because on the right and on the left hand sides the same integral terms cancel, the claim follows.
\end{proof}

\begin{proposition}\label{prop.borne.zt.veps.rhoe}
 Under Assumptions \ref{ass.zeta}, \ref{hypo.data.deux}  and \ref{ass.beta}, let 
 $(\rhow_w,w)$ be the solution of the fully coupled and stabilized problem  \eqref{eq.rho.mod}-\eqref{eq.w.mod}-\eqref{eq.g.mod},
   there exists a positive finite constant $\gamma_2$ s.t. 
 $$
 \int_{\rr} \zteps(w(t,a)) |w(t,a) | \rhow_w(t,a) da \leq \frac{\gamma_2}{\umu} ,\quad \forall t \geq 0 \; ,
 $$ 
 where the constant $\gamma_2$ is depends on
\begin{itemize}
 \item  the {\em a priori} bound only on $\int_{\rr} \rhow_w |w| da$ (obtained in Corollary \ref{coro.stab.rho.u.mod}) ,
 \item  $\nrm{\dt f}{L^\infty(0,T)}$, 
 \item $\ztlip$, and  $\zeta(0)$.
 \end{itemize}
\end{proposition}

\begin{proof}
 Using equations \eqref{eq.rho.mod}, \eqref{eq.w.mod} and hypotheses \ref{ass.zeta}, one has
$$
\e \dt (\rhow_w | w | \zteps )  + \da (\rhow_w | w | \zteps ) +  \zteps^2 | w | \rhow_w \leq 
 \rhow_w |w| ( \e \dt \zteps + \da \zteps )  
 +  \zteps \rhow_w | \tig_w | \; .
$$
 Integrating in age and setting $p(t):=\int_{\rr} \rhow_w(t,a) | w(t,a) | \zteps(w(t,a)) \, da$ gives 
$$
\begin{aligned}
  \e \dt p + \int_{\rr} \zteps^2  | w(t,a) | \rhow_w(t,a)  da & \leq  |\tig_w| \left(  \ztlip\int_{\rr} \rhow_w |w| \, da
  + \int_{\rr} \zeta(w) \rhow_w(t,a) da \right) \\
  & \leq  |\tig_w| \left(  2 \ztlip\int_{\rr} \rhow_w |w| \, da  + \zeta(0) \right) \\
  &  \leq \frac{1}{\umu}( \e | \dt f | + p ) \left(  2 \ztlip / \gamma_1  + \zeta(0) \right), 
\end{aligned}
$$
where $\int_{\rr} \rhow_w | w | da \leq 1/\gamma_1$. 
Now we consider the second term in the left hand side above: using Jensen's inequality one writes
$$
\left( \frac{ \int_{\rr} \zeta(w) | w(t,a) | \rhow_w(t,a) da}{\int_{\rr} | w | \rhow_w da }\right)^2 \leq \frac{ \int_{\rr} (\zteps(w))^2 | w(t,a) | \rhow_w(t,a) da}{\int_{\rr} | w | \rhow_w da} \;,
$$
since $| w| \rhow_w / \int_{\rr} |w| \rhow_w \, da $ is a unit measure. This implies that
$$
\int_{\rr} (\zteps(w))^2 | w(t,a) | \rhow_w(t,a) \; da 
\geq \frac{ \left( \int_{\rr} \zeta(w) | w(t,a) | \rhow_w(t,a) da \right)^2 }{\int_{\rr}  |w| \rhow_w da } 
\geq \gamma_1 p^2 \; .
$$
We obtain a Riccati inequality
$$
\e \dt p + \gamma_1 p^2 \leq h/\umu +  p/ \umu \;,\quad p(0) = \int_{\rr} \zteps(\vepsi(a))|\vepsi(a)|\rhoi(a)\; da \;, 
$$
where $h:= \e  \nrm{\dt f}{\infty} \left(  2 \ztlip / \gamma_1  + \zeta(0) \right)$ is a constant. 
We denote by $P_\pm$ the  solutions 
of the steady state equation associated to the last inequality, {\em i.e.} $P$ solves  $\gamma_1 P^2 -  P/\umu -h/\umu = 0$. 
The solutions are given by 
$$
P_\pm = \frac{1}{\umu}\left(  1 \pm \sqrt{1  + 4 h \umu \gamma_1}\right)/(2\gamma_1) \leq \frac{1}{\umu}\max\left( p(0), \left(  1 \pm \sqrt{1  + 4 h  \gamma_1}\right)/(2\gamma_1)\right)  =: \frac{\gamma_2}{\umu}.
$$
 Applying Lemma \ref{lem.ric},
we conclude that $p(t) \leq \max\{p(0), P_+\}\leq \gamma_2/\umu$,
which ends the proof.
\end{proof}

\begin{theorem}\label{thm.simple.cut.off}
 Suppose that Assumptions \ref{ass.zeta}, \ref{hypo.data.deux} and \ref{ass.beta} hold, 
 moreover, suppose that $\vepsi \in L^\infty(\rr,\omega)$ and that $\nrm{\dt f}{L^\infty(0,T)}$ is finite,
 if  $(\rhow_w,w)$ is  the unique  solution of the stabilized problem  \eqref{eq.rho.mod}-\eqref{eq.w.mod}-\eqref{eq.g.mod},
 it is also the unique solution of   \eqref{eq.rho.mod}-\eqref{eq.w.mod} together with the modified right hand side :
\begin{equation}\label{eq.cut.off.simple}
\titig_w =  \frac{1}{\max(\mu_{0,w}, \umu)}\left( \e \dt f + \int_{\rr}  \zeta(w) w \rhow_w \,  da \right)  \; .
\end{equation}
\end{theorem}

\begin{proof}
 The proof is a simple application of the Proposition \ref{prop.borne.zt.veps.rhoe}
 above and taking $\ovp > \gamma_2 /\umu$ when solving 
  \eqref{eq.rho.mod}-\eqref{eq.w.mod}-\eqref{eq.g.mod}. Indeed, in this case one has that
  the truncated right hand side from \eqref{eq.g.mod} becomes \eqref{eq.cut.off.simple},  one since $p(t) := \int_{\rr} \zeta(w) w \rhow_w da$ 
  never reaches  $\pm \ovp$. 
\end{proof}

\section{Impact of the cut-off value on the mean bonds' population
}\label{sec.intermediate}

In this section we give ourselves a function $g\in L^\infty(0,T)$ and compute $w\in X_T$ solving \eqref{eq.w.eps}.
In what follows we analyze the properties of an age structured model for linkages whose off-rates depend on $w$~:
we define $\rhow_w$ as the solution of 
\begin{equation}
\label{eq.rho.w} 
\left\{ 
\begin{aligned} 
&\varepsilon \partial_t \rhow_w + \partial_a \rhow_w+ \zeta(w) \, \rhow_w = 0 
\,, &t>0 \, , \;a > 0 \; , \\ 
& \rhow_w(t,a=0)=\beps(t)\left(1-\mu_{0,w}  \right) 
\, , & t > 0 \; , \\
& \rhow_w(t=0,a)=\rhoi(a)
\, ,&a\geq 0  \; , 
\end{aligned}  
\right. 
\end{equation}
where $\mu_{0,w}(t):=\int_{\rr }\rhow_w(t,\tia)d\tia$.


We compute  a sharper upper bound on $\mu_{0,w}$, namely
\begin{lemma}\label{lem.moy.zt}
Let Assumptions \ref{ass.zeta} and \ref{ass.beta} hold.
Let $\rhow_w$ be the solution of \eqref{eq.rho.w}.
We suppose that $\mu_{0,w}(0)<1$.
Let us fix  a positive constant $\gamma_0$ s.t. 
$$
\gamma_0 < \min\left( 1-\mu_{0,w}(0) , \frac{ \ztmin}{ \ztmin + \bmax }  \right).
$$
 Under Assumptions \ref{ass.zeta}, \ref{hypo.data.deux}  and \ref{ass.beta}, 
 $\mu_{0,w}(t)< 1-\gamma_0$ holds for every positive time $t$.
\end{lemma}
\begin{proof}
 We proceed similarly as in Lemma 2.2 in \cite{MiOel.1}.
 The computations are thus only formal although  
 they can be made rigorous exactly as therein. 
 By hypothesis, the  data satisfies $1-\gamma_0-\mu_{0,w}(0) >0$. 
 By continuity this also holds on a  time interval  
 $[0,t_0)$ small enough. 
 We proceed by contradiction and suppose that at time $t_0$ the mass $\mu_{0,w}(t_0)$ reaches $1-\gamma_0$.
 The equation on $\mu_{0,w}$ reads:
 $$
 \e \dt \mu_{0,w} - \beps (1- \mu_{0,w}) + \int_{\rr} \rhow_w(t,a) \zeta(w(t,a)) da = 0\;.
 $$
Multiplying it by $-1$ and estimating $\beps(t)\leq\bmax$, one deduces that
 $$
 \e \dt (1- \gamma_0 - \mu_{0,w})+ \bmax (1 - \gamma_0 - \mu_{0,w} ) + \gamma_0 \bmax - \int_{\rr} \rhow_w(t,a) \zeta(w(t,a)) da \geq 0 \; ,
 $$
 then  the lower bound on $\zteps$ implies
 $$
 \e \dt (1- \gamma_0 - \mu_{0,w})+ \bmax (1 - \gamma_0 - \mu_{0,w} ) + \gamma_0 \bmax  \geq \ztmin \mu_{0,w} \;.
 $$
 We transform the latter right hand side writing
 $$
 \ztmin \mu_{0,w} = -\ztmin (1-\gamma_0 - \mu_{0,w}) + \ztmin (1-\gamma_0) \; .
 $$
 Setting $q(t):=(1-\gamma_0 - \mu_{0,w}(t))$, one then has
 $$
 \e \dt q + (\ztmin+ \bmax) q \geq \ztmin - (\ztmin+ \bmax) \gamma_0 > 0 \; ,
 $$
the latter estimate being true under the hypothesis that $\gamma_0 < \ztmin/(\ztmin+ \bmax)$. 
The conclusion then follows integrating the latter inequality in time
$$
q(t_0) >  \exp\left( - \frac{(\bmax+\ztmin)t_0}{\e} \right) q(0) > 0,
$$
under the hypothesis that $\gamma_0 < (1-\mu_{0,w}(0))$. 
But this contradicts the assumption that $q(t_0)=0$, which ends the proof.
\end{proof}

We do not have  a positive definite lower bound on $\mu_{0,w}$ yet ~:
at this stage we only know that $\mu_{0,w}(t)\geq 0$. 
For this reason we define
$\trhoed(t,a) := \rhow(t,a)/(\mu_{0,w}(t)+ \delta)$ and 
we observe that this new function is in $L^\infty_\loc((0,T)\times\rr)$. It solves the equation
\begin{equation}
\label{eq.trho.eps} 
\left\{ 
\begin{aligned} 
&\varepsilon \partial_t \trhoed + \partial_a \trhoed
 + \left( \zteps - \int_{\rr} \zteps \trhoed \right) \trhoed \\
&\hspace{1.95cm}+ \beps\left( \frac{1}{\mu_{0,w} + \delta} - \frac{\mu_{0,w}}{\mu_{0,w}+\delta}\right) \trhoed = 0 
\,, &t>0 \, , \;a > 0 \; , \\ 
& \trhoed(t,a=0)=\beps(t)\left(\frac{1}{\mu_{0,w}+\delta}- \frac{\mu_{0,w}}{\mu_{0,w}+\delta}\right)\, , & t > 0 \; , \\
& \trhoed(t=0,a)=\rhoi(a) / ( \mu_{0,w}+ \delta) \, ,&a\geq 0  \; . 
\end{aligned}  
\right. 
\end{equation}
When $\delta=0$ one denotes $\trhoed$ simply by $\trhow$.
\begin{proposition}\label{prop.moy.zt}
Let $g\in L^\infty(0,T)$ be given, and let $(\rhow,w)$ be the solutions of \eqref{eq.rho.w}-\eqref{eq.w.eps}.
 Under Assumptions \ref{ass.zeta}, \ref{hypo.data.deux}  and \ref{ass.beta} and if $\mu_{0,w}(0) \leq 1 - \gamma_0$, 
  there exists a constant $\bar \zeta$  independent of $\delta$ and $\e$ such that for every positive $\delta$ it
holds that
 $$
 \int_{\rr} \zeta(w(t,a)) \, \trhoed(t,a) da \leq \bar \zeta + \ztlip  \, \nrm{g}{L^\infty(0,T)} \, \min\left(\frac{2}{\gamma_0 \bmin}, \frac{T}{\e}\right) \; ,\quad \forall t \geq 0,
 $$
where
$
\bar \zeta := \zeta(0)+ \int_{\rr} \zeta(\vepsi(a)) \ti{\rho}_{\e,I}  (a) da
$.
Taking the limit as $\delta$ goes to $0$, one obtains then the analogous result for $\trhow$. 
\end{proposition}
\begin{proof}
The product $p(t,a):=\zteps \trhoed$ satisfies
$$
\e \dt p + \da p  + \left( \zteps^2  - \zteps \int_{\rr} \zteps \trhoed \right) \trhoed +  \trhoed(t,0)  p =
\zeta'(w) g(t) \trhoed \;.
$$
Indeed, using arguments as in Lemma 2.1 p. 489 and Lemma 3.1 p. 493 \cite{MiOel.1},  
one proves that if $w$ solves \eqref{eq.w.eps} and $\zeta$ is Lipschitz, then
$\zeta(w)$ solves $(\e \dt + \da) \zeta(w) = \zeta'(w)g$ in the  sense
of characteristics (as in Theorem \ref{prop.rho.exist}) with the corresponding boundary conditions. 
Then  the latter equation on $p$ is understood in the same manner.  

Integrating in age and setting $q(t):=\int_{\rr} p(t,a) da$ we conclude that
\begin{equation}\label{ineq.q}
 \e \dt q  - \zteps(t,0) \trhoed(t,0)   +  \int_{\rr}  \zteps^2 \trhoed  da -  \left( \int_{\rr} \zteps \trhoed \right)^2 
 + q \trhoed(t,0) \leq \ztlip \nrm{g}{\infty} \, .
\end{equation}
To find a lower bound for $\trhoed(0,t)$ we choose $\delta < \gamma_0/2$ and use the upper bound on $\mu_{0,w}(t)$ established in Lemma \ref{lem.moy.zt} in order to obtain 
\begin{equation}
\label{inequ_rhotildedelta}
\trhoed(0,t) \geq \bmin \left( \frac{1}{ 1 - \gamma_0 + \delta } - 1  \right) \geq \bmin \frac{\gamma_0}{2}  \;.
\end{equation}
Assuming  $\mu_{0,w}(t) > 0$ we also find using  Jensen's inequality that
$$
 \left( \int_{\rr} \zeta(w(t,a)) \trhoed(t,a) da \right)^2 
 \leq \int_{\rr} (\zteps(w(t,a)))^2 \trhoed da \frac{\mu_{0,w}}{(\mu_{0,w}+\delta)} \leq \int_{\rr} (\zteps(w(t,a))) ^2 \trhoed da \;.
$$
If $\mu_{0,w}(t)=0$ the same inequality holds true since then $\rhow(t,a)=0$ for almost every $a$. These considerations 
allow then to rewrite \eqref{ineq.q} as 
$$
\e \dt q  + \trhoed(0,t) (q-  \zeta(0)) \leq  \ztlip \nrm{g}{\infty}\;.
$$
Setting $\ti{q}:= q - \zeta(0)$ and using Gronwall's Lemma gives  
$$
\ti{q}(t) \leq \exp\left( - \ue \int_0^t \trhoed(0,s) ds \right) \ti{q}(0) + \frac{\ztlip \nrm{g}{\infty}}{\e} \int_0^t \exp\left( - \ue \int_\tau^t \trhoed(0,s) ds \right) d \tau \; .
$$
Thanks to the uniform lower bound \eqref{inequ_rhotildedelta} we conclude
$$
\ti{q}(t) \leq \exp\left( - \frac{ \bmin \gamma_0 t }{2 \e}\right)  \ti{q}(0) + \frac{2 \ztlip \nrm{g}{\infty}}{\gamma_0 \bmin} \left( 1 - \exp\left( - \frac{\bmin \gamma_0 t}{2 \e} \right) \right)\;,
$$
which then gives turning to the variable $q$ that
\begin{equation}\label{eq.lambda.t}
 q(t) 
\leq \zeta(0)+ \int_{\rr} \zeta(\vepsi(a)) \ti{\rho}_{\e,I}  (a) da + \frac{2 \ztlip \nrm{g}{\infty}}{\gamma_0 \bmin} \left( 1 - \exp\left( - \frac{\bmin \gamma_0 t}{2 \e} \right) \right) \; . 
\end{equation}
This bound is uniform in $\delta$. One  passes to the limit $\delta=0$ which gives the final result.
\end{proof}
%
%

\begin{proposition}\label{prop.borne.inf.muze}
 Under the same assumptions as above and  if $\mu_{0,w}(t) < 1 - \gamma_0$ and choosing $\mumin$ s.t.
 $$
 \mumin < \min \left( \mu_{0,w}(0) , \frac{ \bmin }{\bmin + \bar \zeta + \ztlip  \, \nrm{g}{L^\infty(0,T)} \, \min\left(\frac{2}{\gamma_0 \bmin}, \frac{T}{\e}\right) }\right) \; ,
 $$
where we used the bound provided by Proposition \ref{prop.moy.zt},  one has a lower bound on $\mu_{0,w}$~:
  $$
 \mu_{0,w}(t) \geq \mumin \; , \quad \forall t \geq 0 \; .
 $$
\end{proposition}

\begin{proof}
 We integrate the equation \eqref{eq.rho.eps} with respect to age which gives
 $$
 \left\{
 \begin{aligned}
&  \e \dt \mu_{0,w} - \beps (1- \mu_{0,w}) + \int_{\rr} \rhow_w(t,a) \zeta(w(t,a)) da = 0 \;, & t>0 \; , \\
& \mu_{0,w}(0)=\int_{\rr} \rhoi(a)da\; ,& t=0 \;.
\end{aligned}
\right.
 $$
 we divide and we multiply the last term on the left hand side by $\mu_{0,w}$ and we write:
 $$
 \e \dt \mu_{0,w} + \left( \beps (t)+ \int_{\rr} \zteps \trhow da \right)\mu_{0,w} = \beps(t)\;,\quad  t>0\;.
 $$
 Now suppose that there exists a time small enough s.t. $\mu_{0,w}(t) > \mumin$ for all $t \in [0,t_0)$ and that $\mu_{0,w}(t_0)=\mumin$. 
We use the notation $\olambda:=\bar \zeta + \ztlip  \, \nrm{g}{L^\infty(0,T)} \, \min\left(\frac{2}{\gamma_0 \bmin}, \frac{T}{\e}\right)$ and write for the difference $\tmuze(t):=\mu_{0,w}(t)-\mumin$
\begin{multline*}
 \e \dt  \tmuze 
+  \left( \beps (t)+ \int_{\rr} \zteps \trhow da \right)  \tmuze  = \beps(t)(1  - \mumin ) -  \left( \int_{\rr} \zteps \trhow da\right) \mumin \geq  \\
\geq  \bmin (1-\mumin) - \olambda \mumin 
=\left(\bmin+ \olambda\right) \left( \frac{\bmin}{\bmin+ \olambda} -\mumin \right) > 0
\;,
\end{multline*}
which holds thanks to the bound on $\int_{\rr} \zteps \trhow \, da$ established in Proposition \ref{prop.moy.zt} and the definition of $\mumin$.
Using   Gronwall's Lemma we finally obtain that
$$
\mu_{0,w}(t_0)-\mumin > \exp\left( - \ue \int_0^{t_0} ( \beps(\tau) + \olambda) d\tau \right) (\mu_{0,w}(0)-\mumin) > 0 \;,
$$
which contradicts the fact  that $\mu_{0,w}(t_0)=\mumin$. This ends the proof.
\end{proof}

\section{Local existence of the fully coupled problem} \label{sec.ex.local}

\begin{theorem}
Let $f$ be a Lipschitz function on $(0,T)$ and $\vepsi \in L^\infty(\rr,\omega)$.
We suppose that Assumptions \ref{ass.zeta}, \ref{hypo.data.deux} and \ref{ass.beta} hold.
 Let $(\rhow_w,w)$ be the solution of \eqref{eq.rho.mod}-\eqref{eq.w.mod} together with  $\titig_w$, the simple cut-off
 defined by \eqref{eq.cut.off.simple}. Then for any fixed $\umu<\mu_{0,w}(0)$ there exists a time
 $$
 T= \frac{\e}{\gamma_3} \left( \bmin \umu - (\bmin+\bar \zeta)\umu^2 \right)
 $$
 for which $\mu_{0,w}(t)>\umu$ for any $t\in(0,T)$. So the solution $(\rhow,w)$ of \eqref{eq.rho.mod}-\eqref{eq.w.mod}-\eqref{eq.cut.off.simple} is also the unique local solution of 
 the fully coupled system \eqref{eq.rho.eps}-\eqref{eq.veps}.
\end{theorem}

\begin{proof}
 Gathering results above one has :
$$
\nrm{\titig_w}{L^\infty(0,T)} \leq \frac{ 1}{\umu} \left( \e | \dt f | + p(t) \right) \leq \frac{ 1}{\umu} \left( \e \nrm{ \dt f}{L^\infty(0,T)}  + \frac{\gamma_2}{\umu} \right) \leq \frac{\gamma_3}{\umu^2},
$$ 
since we suppose that $\umu <1$ and we set $\gamma_3 := \ztlip( \e \nrm{ \dt f}{L^\infty(0,T)}  + {\gamma_2})$.
Thanks to Proposition \ref{prop.borne.inf.muze},  the lower bound on $\mu_{0,w}$ then becomes :
$$
\mu_{0,w}(t) > \min \left( \mu_{0,w}(0) , \frac{\bmin \umu^2}{(\bmin+\bar \zeta)\umu^2 + \frac{\gamma_3 T}{\e}} \right).
$$
Choosing $\umu<\mu_{0,w}(0)$ we tune $T$ s.t.
$$
 \frac{\bmin \umu^2}{(\bmin+\bar \zeta)\umu^2 + \frac{\gamma_3 T}{\e}}  > \umu
 $$
\end{proof}

\section{Global existence for specific data}\label{sec.spec.data}
Under hypotheses of Theorem \ref{thm.sys.mod},
whatever be the  time of existence $T$ for $(\rhow_w,w)$,  
the solutions of the stabilized model, then 
thanks to Corollary \ref{coro.stab.rho.u.mod} one has that : 
$$
\begin{aligned}
 \int_{\rr}  \zeta(w(t,a)) \rhow(t,a) da&  \leq \int_{\rr} (\zeta(0) + \ztlip | w | ) \rhow_w da \\
 & \leq \zeta(0) + \ztlip \left( \int_{\rr} |\vepsi| \rhoi da + \int_0^T | \dt f | ds \right) =: \breve{\zeta}, \quad \forall t \in (0,T).
\end{aligned}
$$
\begin{proposition}\label{prop.mu.pos}
Under assumptions \ref{ass.zeta},  \ref{hypo.data.deux} and \ref{ass.beta},  
 if $\bmin > \breve{\zeta}$ and we set : 
 $$
 0< \mumin < \min \left( 1-\frac{\breve{\zeta}}{\bmin}, \mu_{0,w}(0) \right)
 $$
 then one has
 $$
 \mu_{0,w}(t) \geq \mumin ,\quad \forall t \in (0,T). 
 $$
\end{proposition}
\begin{proof}
We set $\hatmu := \mu_{0,w}(t) - \mumin$ and write the equation that it satisfies~:
$$
\e \dt \hatmu + \beta \hatmu  = -  \int_{\rr} \zeta \rhow da+ \beta(1  - \mumin )  \geq - \breve{\zeta} + \bmin( 1 - \mumin ).
$$
We estimate from below the right hand side using previous bounds.
The lower bound is positive definite provided that $\bmin > \breve{\zeta}$ and that $\mumin < 1 - \breve{\zeta}/\bmin$. Using Gronwall's Lemma, one has :
$$
\hatmu (t) \geq \exp( - \bmax t / \e) \hatmu(0) > 0
$$
if $\mumin < \mu_{0,w}(0)$, which ends the proof.
\end{proof}

\begin{theorem}\label{thm.ex.global}
If we fix a finite time $T>0$.
Under Assumptions \ref{ass.zeta} and  \ref{hypo.data.deux},  and assuming that 
\begin{enumerate}[i)]
 \item $f$ is Lipschitz on $(0,T)$,
 \item $\beps$ satisfies assumptions \ref{ass.beta} together with
 $\bmin > \breve{\zeta}$
\end{enumerate}
there exists a unique solution 
 $(\rhoe,\veps)\in C(0,T;L^1(\rr))\times X_T$ solving system \eqref{eq.rho.eps}-\eqref{eq.veps}. 
\end{theorem}

\begin{proof}
 By Theorem \ref{thm.simple.cut.off}, there exists a unique couple $(\rhow_w,w)\in C(0,\infty;L^1(\rr))\times X_\infty$ solving \eqref{eq.rho.mod}-\eqref{eq.w.mod}-\eqref{eq.cut.off.simple} for any given constant $\umu$.
 We choose $T>0$ and provided that $\beta$ satisfies hypothesis required by Proposition \ref{prop.mu.pos}
  we set the constants $0<\umu <\mumin$ according to  Propositions \ref{prop.mu.pos}. Then    $\mu_{0,w}$ does not reach the threshold value $\umu$  so that 
  $$
\begin{aligned}
\titig_w(t) = & \frac{1}{\max(\mu_{0,w}(t),\umu)} \left( \e \dt f + \int_{\rr} ( \zeta(w) \rhow_w w )(t,a)da \right)
= \\
&  \frac{1}{\mu_{0,w}(t)} \left( \e \dt f +  \int_{\rr} ( \zeta(w) \rhow_w w )(t,a)da \right) = g_w(t), \quad a.e. \; t \in (0,T).
\end{aligned}
$$
The pair $(\rhow_w,w)$ is in fact also solving \eqref{eq.rho.eps}-\eqref{eq.veps} on this time interval.
This provides existence of a solution $(\rhoe,\veps)=(\rhow_w,w)$ on $[0,T]$.
Since  by Theorem \ref{thm.simple.cut.off} $(\rhow_w,w)$  is unique,  so is $(\rhoe,\veps)$ in this time period.
\end{proof}

\section{Blow up  for positive solutions}\label{sec.blow.up}

\begin{theorem}\label{thm.positivity.veps}
Under assumptions \ref{hypo.data.deux} and
 if $T_0$ is the  time of existence of $(\rhoe,\veps)$ solving \eqref{eq.rho.eps}-\eqref{eq.veps}, and if
 \begin{enumerate}[i)]
 \item $\vepsi(a)\geq0$ for a.e. $a\in \rr$,
 \item $\dt f(t) >0$ for a.e. $t\in(0,T_0)$, 
\end{enumerate}
then the product $\rhoe(t,a) \veps(t,a)$ 
is non-negative for a.e. $(t,a) \in (0,T_0)\times\rr$.
\end{theorem}
\begin{proof}
Since it holds that $f(0) = \int_{\rr} \rhoi (a)  \vepsi(a) \, da$ and
$f(t) = \int_{\rr} \rhoe (t,a)  \veps(t,a) \, da$ 
yields
\begin{multline*}
\int_{\rr} \rhoe (t,a) | \veps(t,a) | da \leq 
 \int_{\rr} \rhoi(a) | \vepsi(a) | da + \int_0^t | \dt f(\tit)  | d\tit=\\=
 \int_{\rr} \rhoi(a)  \vepsi(a)  da + \int_0^t  \dt f(\tit)  \, d\tit
=
 f(t) = \int_{\rr} \rhoe (t,a)  \veps(t,a) \, da \, ,
\end{multline*}
which implies the result.
\end{proof}

\begin{lemma}\label{lem.conv}
 Suppose that $\zeta_c$ is a convex locally differentiable function. 
 Then for any function $u \in X_\infty$, one  has : 
 $$
 \zeta_c'(0) \int_{\rr} u \rhoe da \leq \int_{\rr} \zeta_c(u(t,a))\rhoe(t,a) da - \zeta_c(0) \muze,\quad \text{ a.e. } t\in \rr
 $$
\end{lemma}

\begin{proof}
 Since $\zeta_c$ is convex, for almost every $(t,a)$, 
 one has ~:
 $$
 \zeta_c'(0)(u(t,a)-0) \leq \zeta_c(u(t,a))-\zeta(0)
 $$
 and integrating with respect to $\rhoe da$, one has the desired result.
\end{proof}

\begin{proposition}
 Under assumptions \ref{hypo.data.deux} and \ref{ass.beta} and if 
\begin{enumerate}[i)]
\item $\zeta$ satisfies Assumptions \ref{ass.zeta}
and admits a lower convex envelop $\zeta_c$ s.t. $\zeta_c(u)\leq \zeta(u)$ for all $u \in \rr$ with 
$\zeta_c'(0)>0$,
\item let $f$ be a Lipschitz function s.t. $\dt f (t) >0$ for a.e. $t \in (0,T)$, 
\item $f$ and $\beta$ are  s.t. $\bmax<\zeta'_c(0)\fmin$,
\item $\vepsi(a)\geq0$ for a.e. $a\in\rr$, 
\end{enumerate}
then 
if the solution $(\rhoe,\veps)$ solving \eqref{eq.rho.eps}-\eqref{eq.veps}  exists until a finite time $T_0$, 
this time cannot be greater than 
$$
t_0 := \frac{\e}{\bmin+\zeta_c(0)} \ln\left( 1 + \frac{\muze(0)(\bmin+\zeta_c(0))}{\zeta_c'(0)\fmin-\bmax}\right) 
$$
for which
$$
\muze(t_0)\leq 0.
$$
Moreover, on $(0,t_0)\times\rr$,  one has  a lower bound on the profile of $\veps$ namely 
$$
\veps(t,a) \geq \e \gamma_6 \ln\left( 1 + \frac{ \min(t,\e a)}{(t_0-t)} \right),
$$
where $\gamma_6 := t_0 \inf_{t\in(0,t_0)} \dt f / \muze(0)$. 
\end{proposition}
\begin{proof}
By Theorem \ref{thm.positivity.veps} $\veps(t,a) \geq 0$ a.e. $(t,a) \in(0,T_0)\times\rr$.
 The equation for $\muze$ reads :
 $$
 \e \dt \muze - \beta (1-\muze) + \int_{\rr} \zeta(\veps(t,a)) \rhoe(t,a) da = 0
 $$
 that we estimate using Lemma \ref{lem.conv} as follows :
 $$
 \e \dt \muze - \beta (1-\muze)  + \zeta_c'(0) \int_{\rr} \veps(t,a) \rhoe(t,a) da + \zeta_c(0) \muze \leq  0
 $$
 and becomes under these simplifications :
\begin{equation} \label{eq.simp.mu}
  \e \dt \muze - \beta (1-\muze) + \zeta_c(0) \muze + \zeta_c'(0) f \leq 0.
\end{equation}
We can deduce  from this equation that
$$
\e \dt \muze + (\bmin +\zeta_c(0)) \muze \leq \bmax - \zeta_c'(0) \fmin
$$
which gives using Gronwall's Lemma that $ \muze(t) \leq \oomu(t)$, where 
$$
\oomu(t):= \muze(0) \exp\left( - \frac{(\bmin+\zeta_c(0))}{\e}t \right) - \frac{\zeta_c'(0)\fmin-\bmax}{(\bmin+\zeta_c(0))} \left(1- \exp\left( - \frac{(\bmin+\zeta_c(0))}{\e}t \right)\right).
$$
Looking for the time $t_0$ s.t. $\oomu(t_0)=0$ provides  the explicit form of $t_0$ in the claim. Thus $T_0<t_0$. 
Moreover, as $\oomu(t)$ is a convex function
one has that :
$$
\muze(t)\leq \left(1-\frac{t}{t_0}\right)\oomu(0) + \frac{t}{t_0} \oomu(t_0) \equiv \left(1-\frac{t}{t_0}\right)\oomu(0), 
$$
and because,  by Lemma \ref{lem.conv}, $\zeta(\veps) \veps \rhoe $ is positive almost everywhere on $(0,t_0)\times \rr$, 
$$
\e \dt \veps + \da \veps \geq \frac{\e \dt f}{\muze(t)} \geq \frac{\e \gamma_6}{t_0-t} , \quad \text{ a.e in } (0,t_0)\times \rr.
$$ 
Using  Duhamel's formula provides 
$$
\veps(t,a) \geq 
\begin{cases}
 \e \gamma_6 \int_{-a}^0 \frac{ds}{t_0 -(t+\e s)} ds , & \text{if } t\geq \e a ,\\
  \vepsi(a-t/\e) + \e \gamma_6 \int_{-t /\e}^0 \frac{ds}{t_0 -(t+\e s)} ds & \text{ otherwize},
\end{cases}
$$
which then gives the lower estimate on $\veps$. 
\end{proof}

\appendix
\section{Riccati inequalities}

\begin{lemma}\label{lem.ric}
 Let $\e>0$ and real, let $y$ be a positive differentiable function of $t\in \rr$, satisfying 
$$
\left\{
 \begin{aligned}
&  \e \dt y  + A y^2  \leq  B y + C, & t>0, \\
& y(0)=y_0,& t=0 
\end{aligned}
\right.
$$
where $y_0>0$ and $(A,B,C) \in (\rr)^3$. Setting $y_+ := (B+\sqrt{B^2+4 A C} )/(2A)$, one has that
$$
y(t) \leq \max(y_0,y_+),\quad \forall t \in \rr.
$$
\end{lemma}
\begin{proof}
 We set  $m:=\max(y_0,y_+)$,  it satisfies $-A m^2 + B m + C \leq 0$.
 Then we define $ \ty := y-m$ which then solve the differential inequality :
\begin{equation}\label{eq.ineq}
 \e \dt \ty+ A\ty^2 + (2 m A - B) \ty \leq 0,
\end{equation}
Since the quadratic term is positive we neglect it, and apply Gronwall's Lemma :
$$
\ty(t) \leq \exp\left(-\frac{(2 A m - B )t }{\e} \right) \ty(0) = \exp\left(-\frac{(2 A m - B) t}{\e} \right) \left( y_0 - m\right) \leq 0
$$
which ends the proof.
\end{proof}
\bibliographystyle{plain}
\bibliography{biblio}

\def\cprime{$'$} \def\cprime{$'$}
\begin{thebibliography}{10}

\bibitem{Baumgartner11042000}
W.~Baumgartner, P.~Hinterdorfer, W.~Ness, A.~Raab, D.~Vestweber, H.~Schindler,
  and D.~Drenckhahn.
\newblock Cadherin interaction probed by atomic force microscopy.
\newblock {\em Proceedings of the National Academy of Sciences},
  97(8):4005--4010, 2000.

\bibitem{canetta:inserm-00144609}
E.~Canetta, A.~Duperray, A.~Leyrat, and C.~Verdier.
\newblock {Measuring cell viscoelastic properties using a force-spectrometer:
  influence of protein-cytoplasm interactions.}
\newblock {\em {Biorheology}}, 42(5):321--33, 2005.

\bibitem{pmid8432732}
F.~Gittes, B.~Mickey, J.~Nettleton, and J.~Howard.
\newblock {{F}lexural rigidity of microtubules and actin filaments measured
  from thermal fluctuations in shape}.
\newblock {\em J. Cell Biol.}, 120(4):923--934, Feb 1993.

\bibitem{pmid8282102}
W.~H. Goldmann and G.~Isenberg.
\newblock {{A}nalysis of filamin and alpha-actinin binding to actin by the
  stopped flow method}.
\newblock {\em FEBS Lett.}, 336(3):408--410, Dec 1993.

\bibitem{hanley2003single}
W.~Hanley, O.~McCarty, S.~Jadhav, Y.~Tseng, D.~Wirtz, and K.~Konstantopoulos.
\newblock Single molecule characterization of p-selectin/ligand binding.
\newblock {\em Journal of Biological Chemistry}, 278(12):10556--10561, 2003.

\bibitem{pmid18278037}
S.~A. Koestler, S.~Auinger, M.~Vinzenz, K.~Rottner, and J.~V. Small.
\newblock {{D}ifferentially oriented populations of actin filaments generated
  in lamellipodia collaborate in pushing and pausing at the cell front}.
\newblock {\em Nat. Cell Biol.}, 10(3):306--313, Mar 2008.

\bibitem{pmid12547805}
F.~Li, S.~D. Redick, H.~P. Erickson, and V.~T. Moy.
\newblock {{F}orce measurements of the alpha5beta1 integrin-fibronectin
  interaction}.
\newblock {\em Biophys. J.}, 84(2 Pt 1):1252--1262, Feb 2003.

\bibitem{Li2003}
F.~Li, S.~D. Redick, H.~P. Erickson, and V.~T. Moy.
\newblock Force measurements of the $\alpha$5$\beta$1 integrin–fibronectin
  interaction.
\newblock {\em Biophysical Journal}, 84(2):1252 -- 1262, 2003.

\bibitem{MiOel.1}
V.~Mili{\v{s}}i{\'c} and D.~Oelz.
\newblock On the asymptotic regime of a model for friction mediated by
  transient elastic linkages.
\newblock {\em J. Math. Pures Appl. (9)}, 96(5):484--501, 2011.

\bibitem{MiOel.2}
V.~Mili{\v{s}}i{\'c} and D.~Oelz.
\newblock On a structured model for the load dependent reaction kinetics of
  transient elastic linkages.
\newblock {\em SIAM J. Math. Anal.}, 47(3):2104--2121, 2015.

\bibitem{OeSch}
D.~Oelz and C.~Schmeiser.
\newblock Derivation of a model for symmetric lamellipodia with instantaneous
  cross-link turnover.
\newblock {\em Archive for Rational Mechanics and Analysis}, 198(3):963--980,
  2010.
\newblock cited By 3.

\bibitem{OeSchVi}
D.~Oelz, C.~Schmeiser, and V.~Small.
\newblock Modelling of the actin-cytoskeleton in symmetric lamellipodial
  fragments.
\newblock {\em Cell Adhesion and Migration}, 2:117--126, 2008.

\bibitem{pmid1093925}
Y.~Osterg, S.~V. Noorden, and A.~G. Pearse.
\newblock {{C}ytochemical, immunofluorescence, and ultrastructural
  investigations on polypeptide hormone localization in the islet parenchyma
  and bile duct mucosa of a cyclostome, {M}yxine glutinosa}.
\newblock {\em Gen. Comp. Endocrinol.}, 25(3):274--291, Mar 1975.

\bibitem{PreVi}
L.~Preziosi and G.~Vitale.
\newblock A multiphase model of tumor and tissue growth including cell adhesion
  and plastic reorganization.
\newblock {\em Math. Models Methods Appl. Sci.}, 21(9):1901--1932, 2011.

\bibitem{Suda_2001}
H.~Suda.
\newblock Origin of friction derived from rupture dynamics.
\newblock {\em Langmuir}, 17(20):6045--6047, 2001.

\bibitem{pmid16183875}
M.~Sun, J.~S. Graham, B.~Hegedus, F.~Marga, Y.~Zhang, G.~Forgacs, and
  M.~Grandbois.
\newblock {{M}ultiple membrane tethers probed by atomic force microscopy}.
\newblock {\em Biophys. J.}, 89(6):4320--4329, Dec 2005.

\end{thebibliography}

     \end{document}